\documentclass[12pt]{amsart}
\usepackage{amssymb,amsmath,amscd,enumerate,verbatim}
\usepackage[all]{xy}
\usepackage{fullpage}
\usepackage{pdfsync}

\numberwithin{equation}{section}

\newtheorem{theorem}{Theorem}[section]
\newtheorem{lemma}[theorem]{Lemma}
\newtheorem{proposition}[theorem]{Proposition}
\newtheorem{corollary}[theorem]{Corollary}

\theoremstyle{definition}

\newtheorem{def-prop}[theorem]{Definition-Proposition}

\newtheorem{remark}[theorem]{Remark}
\newtheorem{example}[theorem]{Example}

\newtheorem*{acknowledgement}{Acknowledgement}

\DeclareMathOperator{\chara}{char}

\DeclareMathOperator{\Tor}{Tor}
\DeclareMathOperator{\reg}{reg}
\DeclareMathOperator{\depth}{depth}

\DeclareMathOperator{\Ass}{Ass}
\DeclareMathOperator{\Min}{Min}

\DeclareMathOperator{\spec}{Spec}

\newcommand{\ZZ}{{\mathbb Z}}

\newcommand{\QQ}{{\mathbb Q}}

\def\mm{{\mathfrak m}}

\def\pp{{\mathfrak p}}
\def\qq{{\mathfrak q}}

\begin{document}

\title{Symbolic powers of sums of ideals}

\author{Huy T\`ai H\`a}
\address{Tulane University \\ Department of Mathematics \\
6823 St. Charles Ave. \\ New Orleans, LA 70118, USA}
\email{tha@tulane.edu}

\author{Hop Dang Nguyen}
\address{Institute of Mathematics \\ Vietnam Academy of Science and Technology, 18 Hoang Quoc Viet \\ Hanoi, Vietnam}
\email{ngdhop@gmail.com}

\author{Ngo Viet Trung}
\address{International Centre for Research and Postgraduate Training, Institute of Mathematics \\ Vietnam Academy of Science and Technology, 18 Hoang Quoc Viet \\ Hanoi, Vietnam}
\email{nvtrung@math.ac.vn}

\author{Tran Nam Trung}
\address{Institute of Mathematics, VAST, 18 Hoang Quoc Viet, Hanoi, Viet Nam, and TIMAS, Thang Long University, Nghiem Xuan Yem road, Hoang Mai district, Hanoi, Vietnam}
\email{tntrung@math.ac.vn}

\keywords{Symbolic power, sum of ideals, associated prime, tensor product, binomial expansion, depth, Castelnuovo-Mumford regularity, Tor-vanishing, depth function}
\subjclass[2010]{Primary 13C05, 14B05; Secondary 13D07, 18G15}
		
\begin{abstract}
Let $I$ and $J$ be nonzero ideals in two Noetherian algebras $A$ and $B$ over a field $k$.
Let $I+J$ denote the ideal generated by $I$ and $J$ in $A\otimes_k B$.
We prove the following expansion for the symbolic powers:
$$(I+J)^{(n)} = \sum_{i+j = n} I^{(i)} J^{(j)}.$$
If $A$ and $B$ are polynomial rings and if $\chara(k) = 0$ or if $I$ and $J$ are monomial ideals, we give
exact formulas for the depth and the Castelnuovo-Mumford regularity of $(I+J)^{(n)}$,
which depend on the interplay between the symbolic powers of $I$ and $J$.
The proof involves a result of independent interest which states that under the above assumption, the induced map $\Tor_i^A(k,I^{(n)}) \to \Tor_i^A(k,I^{(n-1)})$ is zero for all $i \ge 0$, $n \ge 0$.
We also investigate other properties and invariants of $(I+J)^{(n)}$ such as the equality between ordinary and symbolic powers, the Waldschmidt constant and the Cohen-Macaulayness. \par
\end{abstract}
\maketitle


\section{Introduction}

Let $R$ be a commutative Noetherian ring and let $Q$ be an ideal in $R$. For an integer $n \ge 1$, the {\em $n$-th symbolic power} of $Q$ is defined by
\[
Q^{(n)} := R \cap \Big(\bigcap_{P\in \Min(Q)} Q^nR_P \Big).
\]
In other words, $Q^{(n)}$ is the intersection of the primary components of $Q^n$ associated to the minimal primes of $Q$.\par

When $R$ is a polynomial ring over an algebraically closed field $k$ of characteristic zero and $Q$ is a radical ideal, Nagata and  Zariski showed that $Q^{(n)}$ consists of polynomials in $R$ whose partial derivatives of orders up to $n-1$ vanish on the zero set of $Q$ (see e.g. \cite{EH}). Therefore, symbolic powers carry more geometric information than ordinary powers of an ideal. In general, it is difficult to study properties of symbolic powers.
\par

Let $A$ and $B$ be commutative Noetherian algebras over an arbitrary field $k$. Let $I \subseteq A$ and $J \subseteq B$ be nonzero proper ideals. For simplicity we also use $I$ and $J$ to denote the extensions of $I$ and $J$ in the algebra $R := A \otimes_k B$. The main aim of this paper is to study the depth and the Castelnuovo-Mumford regularity (or simply regularity) of the symbolic powers of the sum $I+J$ in $R$. Such sums of ideals appear naturally in various contexts:
\begin{itemize}
\item {\it Fiber product of affine schemes}:  Let $X$ and let $Y$ be the affine schemes $\spec(A/I)$ and $\spec(B/J)$, then the fiber product $X \times_k Y$ is the affine scheme $\spec(R/I+J)$;
\item {\it Join of simplicial complexes}:  Let $\Delta$ and $\Gamma$ be simplicial complexes over disjoint vertex sets with Stanley-Reisner ideals $I_\Delta$ and $I_\Gamma$, then $I_\Delta + I_\Gamma$ is the Stanley-Reisner ideal of the join complex $\Delta \star\Gamma$;
\item {\it Edge ideal of a graph}:
Let $I(G)$ denote the edge ideal a simple graph $G$. If $G_1,...,G_n$ are the connected components of $G$, then \par\smallskip
\centerline{$I(G) = I(G_1) + \cdots + I(G_n).$}
\end{itemize}

Though symbolic powers have been studied extensively (see e.g. \cite{Ba,CEHH,CHS, ELS, HaHu, He, HKTT,HHT, HTr, HoHu, MT,Su,TeT,TT,Var}),
symbolic powers of such sums of ideals have not been considered in this general setting.
We shall see from this paper and our sequential work \cite{HNTT2} that studying sums of ideals may indeed provide new insights to many problems on symbolic powers.
\par

Several results on the depth and the regularity of the ordinary powers $(I+J)^n$ have been recently established in \cite{HTT, NgV}. These results have had a number of interesting consequences.
It is quite natural to ask whether there are similar results on the symbolic powers $(I+J)^{(n)}$. \par

The first step is to characterize $(I+J)^{(n)}$ in terms of $I$ and $J$.
In general, if $I$ and $J$ are prime ideals, $I+J$ needs not be a primary ideal.
This indicates that such a characterization would be complicated.
Surprisingly, we can show that there is a binomial expansion for the symbolic power $(I+J)^{(n)}$:
\medskip

\noindent{\bf Theorem \ref{binomial}.}
$(I+J)^{(n)} = \sum_{i+j = n} I^{(i)} J^{(j)}.$
\medskip

This formula was not known even in the simple case when $B = k[x]$ is a polynomial ring and $J = (x)$.
It was known before only for squarefree monomial ideals by Bocci et al \cite{Many}.
The proof of Theorem \ref{binomial} is based on a thorough study of associated primes of tensor products of modules over $A$ and $B$, which is of independent interest. \par

Theorem \ref{binomial} allows us to study several aspects of $(I+J)^{(n)}$.
First, we show that  when $A$ and $B$ are local rings or domains,
$(I+J)^{(n)} = (I+J)^n$ if and only if $I^{(t)} = I^t$ and $J^{(t)} = J^t$  for all $t \le n$, and that when $I$ and $J$ are homogeneous ideals of polynomial rings, then
$$\hat{\alpha}(I+J) = \min \{\hat{\alpha}(I), \hat{\alpha}(J)\},$$
where $\hat{\alpha}(I)$ denotes the Waldschmidt constant of an ideal, which appears in several areas of mathematics \cite{Ch, EV, HaHu, Wa}.
This formula for $\hat{\alpha}(I+J)$ was known before only for squarefree monomial ideals \cite{Many}. \par

Our main results on the depth and the regularity of $(I+J)^{(n)}$ can be summarized as follows.
\medskip

\noindent{\bf Theorem \ref{first bound}}.
{\em Let $A$ and $B$ be polynomial rings over a field $k$.
Let $I \subseteq A$ and $J \subseteq B$ be nonzero proper homogeneous ideals. Then
\begin{enumerate}[\quad \rm(i)]
\item $\depth R\big/(I+J)^{(n)} \ge$
$$\min_{\begin{subarray}{l} i \in [1,n-1]\\ j \in [1,n] \end{subarray}} \left\{\depth A/I^{(n-i)}  +  \depth B/J^{(i)}  +  1, \depth A/I^{(n-j+1)}  +  \depth B/J^{(j)} \right\},$$
\item $\reg R\big/(I+J)^{(n)} \le$
$$ \max_{\begin{subarray}{l} i \in [1,n-1]\\ j \in [1,n] \end{subarray}}  \left\{\reg A/I^{(n-i)} + \reg B/J^{(i)} + 1, \reg A/I^{(n-j+1)} + \reg B/J^{(j)} \right\}.$$
\end{enumerate} }
\medskip

\noindent{\bf Theorems \ref{thm_depthreg_char0monomial} and \ref{monomial}}.
{\em The inequalities of Theorem \ref{first bound} are equalities if $\chara(k) = 0$ or if $I$ and $J$ are monomial ideals.}
\medskip

We expect that equalities hold regardless of the characteristic of the field $k$.

The above results are intricate in the sense that the right-hand sides of the above formulae involve
the minimum or maximum value of two sets of different terms,  which can be attained separately.
It is a distinctive feature of polynomial rings because we can show that they do not hold if one of the rings $A$ and $B$ is  not a polynomial ring.
\par

Using the same approach we further obtain exact formulas for the depth and the regularity of the quotients $(I+J)^{(n)}/(I+J)^{(n+1)}$, that are independent of the characteristic of the field $k$.
\medskip

\noindent{\bf Theorem \ref{main equality}.}
{\em Let $A$ and $B$ be polynomial rings over a field $k$.
Let $I \subseteq A$ and $J \subseteq B$ be nonzero proper homogeneous ideals. Then
\begin{enumerate}[\quad \rm(i)]
\item $\displaystyle \depth (I+J)^{(n)}/(I+J)^{(n+1)} = \min_{i+j=n}\{\depth I^{(i)}/I^{(i+1)} + \depth J^{(j)}/J^{(j+1)}\},$
\item $\displaystyle \reg (I+J)^{(n)}/(I+J)^{(n+1)} = \max_{i+j=n}\{\reg  I^{(i)}/I^{(i+1)} + \reg J^{(j)}/J^{(j+1)}\}.$
\end{enumerate}}
\medskip

As a consequence of Theorem \ref{main equality}, we show that $R/(I+J)^{(i)}$ is Cohen-Macaulay for all $i \le n$ if and only if $A/I^{(i)}$ and $B/J^{(i)}$ are Cohen-Macaulay for all $i \le n$.
\par

The above results hold in a more general framework.
Given two filtrations of ideals $\{I_n\}_{n \ge 0}$ in $A$ and $\{J_n\}_{n \ge 0}$ in $B$, we give bounds for the depth and the regularity of the binomial sum
$$Q_n := \sum_{i+j = n}I_iJ_j.$$
For the filtrations of ordinary or symbolic powers of the ideals $I$ and $J$,  we have $Q_n = (I+J)^n$ or $Q_n = (I+J)^{(n)}$, respectively. This approach can be also applied to filtrations of integral closures or of saturations of powers of $I$ and $J$.
\par

We say that a filtration of ideals $\{I_n\}_{n \ge 0}$ is {\em Tor-vanishing} if
the induced map $\Tor_i^A(k,I_n) \to \Tor_i^A(k,I_{n-1})$ is zero for all $i \ge 0$ and $n \ge 0$.
We show that the bounds for the depth and the regularity of the binomial sum $Q_n$
become equalities if the filtrations $\{I_n\}_{n \ge 0}$ and $\{J_n\}_{n \ge 0}$ are Tor-vanishing.
Theorems \ref{thm_depthreg_char0monomial} and \ref{monomial} follow from this result because
Tor-vanishing holds for filtrations of symbolic powers in these cases:
\medskip

\noindent{\bf Propositions \ref{mapoftors_symbolicpower} and \ref{function_monomial}}.
{\em Let $I$ be a homogeneous ideal in a polynomial ring $A$ over a field $k$.
If $\chara(k) = 0$ or if $I$ is a monomial ideal, the filtration $\{I^{(n)}\}_{n \ge 0}$ is Tor-vanishing.}
\medskip

The Tor-vanishing of the symbolic powers are of independent interest  because they can be used to investigate homological relationships between $I^{(n-1)}$ and $I^{(n)}$ for a homogeneous ideal $I$.
They can be also considered as a higher order generalization of the inclusion $I^{(n)} \subseteq \mm I^{(n-1)}$, where $\mm$ is the ideal generated by the variables of $R$. This inclusion was proved by Eisenbud and Mazur \cite{EM} under the same assumption of Propositions \ref{mapoftors_symbolicpower} and \ref{function_monomial}. Using an example of \cite{EM} and the polarization trick of McCullough and Peeva \cite{MP} we can find homogeneous ideals whose filtration of symbolic powers is not Tor-vanishing if $\chara(k) > 0$. \par
\smallskip

Our paper is structured as follows.
In Section \ref{sect_assoc} we study the associated primes of tensor products of modules.
In Section \ref{sect_binomialexpansion} we prove the binomial expansion of $(I+J)^{(n)}$.
In Section \ref{sect_depthreg}, we present bounds for the depth and the regularity of $R/(I+J)^{(n)}$ and the exact formulas for the depth and the regularity of $(I+J)^{(n-1)}/(I+J)^{(n)}$ in terms of those of $I$ and $J$.
In Section \ref{sect_splittings} we use the technique of Tor-vanishing to study the problem when the obtained bounds for the depth and the regularity of $R/(I+J)^{(n)}$ become exact formulas.  \par

We assume that the reader is familiar with basic properties of associated primes, depth and regularity, which we use without references. For other unexplained notions and terminology, we refer the reader to \cite{BrH, E}.

\begin{acknowledgement}
H.T. H\`a is partially supported by the Simons Foundation (grant \#279786) and Louisiana Board of Regents (grant \#LEQSF(2017-19)-ENH-TR-25).
H.D. Nguyen is supported by a Marie Curie fellowship of the Istituto Nazionale di Alta Matematica. T.N. Trung is partially supported by Project ICRTM01$\_$2019.01 of the International Centre for Research and Postgraduate Training, Institute of Mathematics, VAST and by Vietnam National Foundation for Science and Technology Development (grant \#101.04-2018.307).
Part of this work was done during a research stay of the authors at Vietnam Institute for Advanced Study in Mathematics.
\end{acknowledgement}


\section{Associated primes of tensor products} \label{sect_assoc}

Let $A$ and $B$ be two Noetherian algebras over a field $k$ such that $R := A \otimes_k B$ is Noetherian.
For our investigation on the symbolic powers of sums of ideals, we need to know the  associated primes of $R$-modules of the form $M\otimes_k N$, where $M$ and $N$ are nonzero finitely generated modules over $A$ and $B$.

Let $\Min_A(M)$ and $\Ass_A(M)$ denote the sets of minimal associated primes and associated primes of $M$ as an $A$-module, respectively.  The aim of this section is to describe $\Min_R(M\otimes_k N)$ and $\Ass_R(M\otimes_k N)$ in terms of those of $M$ and $N$.  \par

We begin with the following observations.

\begin{lemma} \label{contraction}
Let $P$ be a prime ideal of $R$, $\pp = P \cap A$, and $\qq = P \cap B$. Then
\begin{enumerate}[\quad \rm(i)]
\item $P \in \Min_R(M \otimes_k N)$ if and only if $\pp \in \Min_A(M)$, $\qq \in \Min_B(N)$ and $P \in \Min_R (R/\pp+\qq)$;
\item $P \in \Ass_R(M \otimes_k N)$ if and only if $\pp \in \Ass_A(M)$, $\qq \in \Ass_B(N)$ and $P \in \Ass_R (R/\pp+\qq)$.
\end{enumerate}
\end{lemma}

\begin{proof}
It is clear that $(M \otimes_k N)_P = M_\pp \otimes_{A_\pp} (A \otimes_k N)_P.$
Since the map $A \to R$ is flat, the map $A_\pp \to R_P$ is also flat.
Applying \cite[Chap. IV, (6.1.2)]{Gr}, we have
$$\dim (M\otimes_k N)_P   =  \dim M_\pp + \dim k(\pp) \otimes_{A_\pp} (A \otimes_k N)_P,$$
where $k(\pp)$ denotes the residue field of $A_\pp$.
Since $k(\pp) = (A/\pp)_\pp$, we have
$((A/\pp) \otimes_k N)_P =  k(\pp) \otimes_{A_\pp} (A \otimes_k N)_P$. Therefore,
$$\dim (M\otimes_k N)_P   =  \dim M_\pp + \dim ((A/\pp) \otimes_k N)_P.$$
Since the map $B \to R$ is flat, we can also show that
$$\dim ((A/\pp) \otimes_k N)_P = \dim N_\qq + \dim ((A/\pp) \otimes_k (B/\qq))_P.$$

Note that $(A/\pp) \otimes_k (B/\qq) = R/\pp +\qq$. From the above equalities we get
$\dim (M\otimes_k N)_P = 0$ if and only if
$$\dim M_\pp = \dim N_\qq = \dim (R/\pp+\qq)_P = 0.$$
For an arbitrary finite $R$-module $E$, we know that $P \in \Min_R(E)$ if and only if $\dim E_P = 0$.
Therefore, $P \in \Min_R(M \otimes_k N)$ if and only if $\pp \in \Min_A(M)$, $\qq \in \Min_B(N)$ and $P \in \Min_R(R/\pp+\qq)$. \par

Similarly, we can apply \cite[Chap. IV, (6.3.1)]{Gr} \label{Gr} to show that $\depth (M\otimes_k N)_P = 0$ if and only if
$$\depth M_\pp = \depth N_\qq = \depth (R/\pp+\qq)_P = 0.$$
We also know that $P \in \Ass_R(E)$ if and only if $\depth E_P = 0$.
Therefore, $P \in \Ass_R(M \otimes_k N)$ if and only if $\pp \in \Ass_A(M)$, $\qq \in \Ass_B(N)$ and $P \in \Ass_R(R/\pp+\qq)$.
\end{proof}

\begin{remark}
We need the assumption on the Noetherian property of $A \otimes_k B$ for applying \cite[Chap. IV, (6.1.2) and (6.3.1)]{Gr}. In general, $A \otimes_k B$ is not necessarily Noetherian, even when $A$ and $B$ are field extensions of $k$. For more information on this topic see \cite{Va}.
\end{remark}

Notice that $\pp +\qq$ is not necessarily a prime or even a primary ideal as illustrated by the following example.

\begin{example} \label{not prime} Let $\pp :=(x^2+1) \subset A = \QQ[x]$ and $\qq := (y^2+1)\subset B = \QQ[y]$. Both $\pp$ and $\qq$ are prime ideals. However, $\pp+\qq = (x^2+1,y^2+1)$ is not primary in $R = \QQ[x,y]$ because $x^2-y^2 = (x+y)(x-y) \in \pp+\qq$.
\end{example}

\begin{lemma} \label{unmixed}
Let $\pp$ and $\qq$ be prime ideals of $A$ and $B$, respectively.
Let $P\in \Ass(R/\pp+\qq)$. Then
\begin{enumerate}[\quad \rm(i)]
\item $P \cap A = \pp$ and $P \cap B = \qq$,
\item $P \in \Min_R(R/\pp+\qq)$.
\end{enumerate}
\end{lemma}

\begin{proof}
Note that $R/\pp+\qq = (A/\pp) \otimes_k (B/\qq)$.
Applying Lemma \ref{contraction} (ii) to $(A/\pp) \otimes_k (B/\qq)$
we obtain $P \cap A \in \Ass_A(A/\pp) = \{\pp\}$ and $P \cap B \in \Ass_B(B/\pp) = \{\pp\}$, which implies (i).

Let $k(\pp)$ and $k(\qq)$ denote the residue fields of $A_\pp$ and $B_\qq$.
Because of (i) we can consider $((A/\pp) \otimes_k (B/\qq))_P$ as
a localization of the algebra $k(\pp) \otimes_k k(\qq)$ at a prime ideal $P'$.
Since $P$ is an associated prime of $(A/\pp) \otimes_k (B/\qq)$,
$P'$ is an associated prime of $k(\pp) \otimes_k k(\qq)$.
By \cite[Theorem 3]{OQ}, $(k(\pp) \otimes_k k(\qq))_{P'}$ is a primary ring, i.e.  any of its zero divisors is a nilpotent element. From this it follows that $P'$ is a minimal associated prime of $k(\pp) \otimes_k k(\qq)$.
Hence, $P$ must be a minimal associated prime of $(A/\pp) \otimes_k (B/\qq)$, which proves (ii).
\end{proof}

By Lemma \ref{unmixed}(ii), the ideal $\pp +\qq$ is always unmixed though it may be not a primary ideal. \par

Now we can describe the associated and the minimal associated primes of $M \otimes_k N$ in terms of $M$ and $N$ as follows. This description gives more precise information on $\Ass_R(M \otimes_k N)$ than \cite[Corollary 3.7(1)]{STY}.

\begin{theorem} \label{asso}
Let $M$ and $N$ be nonzero modules over $A$ and $B$, respectively. Then
\begin{enumerate}[\quad \rm(i)]
 \item $\Min_R(M \otimes_k N) = \displaystyle \bigcup_{\begin{subarray}{l} \pp\in \Min_A(M)\\
 \qq \in \Min_B(N)\end{subarray}} \Min_R(R/\pp+\qq).$
 \item $\Ass_R(M \otimes_k N) = \displaystyle \bigcup_{\begin{subarray}{l} \pp\in \Ass_A(M)\\
 \qq \in \Ass_B(N)\end{subarray}} \Min_R(R/\pp+\qq).$
\end{enumerate}
\end{theorem}

\begin{proof}
By Lemma \ref{contraction}, we have
$$\Min_R(M \otimes_k N)  = \bigcup_{\begin{subarray}{l} \pp\in \Min_A(M)\\
 \qq \in \Min_B(N)\end{subarray}} \{P \in \Min_R(R/\pp+\qq) \big|   P \cap A = \pp, P \cap B = \qq\}.$$
$$\Ass_R(M \otimes_k N)  = \bigcup_{\begin{subarray}{l} \pp\in \Ass_A(M)\\
 \qq \in \Ass_B(N)\end{subarray}} \{P \in \Ass_R(R/\pp+\qq) \big|  P \cap A = \pp, P \cap B = \qq\}.$$
By Lemma \ref{unmixed},  we have $P \cap A = \pp, P \cap B = \qq$ for all $P \in \Ass_R(R/\pp+\qq)$ and
$$\Ass_R(R/\pp+\qq) = \Min_R(R/\pp+\qq).$$
Hence, we can rewrite the above formulas as in the statement of the theorem.
\end{proof}

The following immediate consequence of Theorem \ref{asso} is a generalization of a classical result of Seidenberg  on unmixed polynomial ideals under base field extensions in \cite{Se}.

\begin{corollary}\label{embedded}
$\Ass_R(M \otimes_k N) = \Min_R(M \otimes_k N)$ if and only if $\Ass_A(M) = \Min_A(M)$ and $\Ass_B(N) = \Min_B(N)$.
\end{corollary}

One may ask  when is the sum $\pp + \qq$ of two prime ideals $\pp \subset A$ and $\qq \subset B$ a prime ideal in $R$? This question has the following simple answer.

\begin{lemma}
Let $k(\pp)$ and $k(\qq)$ denote the fields of fractions of $A/\pp$ and $B/\qq$.
Then $\pp + \qq$ is a prime ideal if and only if $k(\pp)\otimes_kk(\qq)$ is a domain.
\end{lemma}

\begin{proof}
Let $\pp+\qq$ be a prime ideal. Then $(A/\pp) \otimes_k (B/\qq) = R/\pp+\qq$ is a domain.
Since $k(\pp)\otimes_kk(\qq)$ is a localization of $(A/\pp) \otimes_k (B/\qq)$, it must be a domain, too. The converse is true since we have an injection $(A/\pp) \otimes_k (B/\qq) \to k(\pp)\otimes_k k(\qq)$.
\end{proof}

By \cite[Corollary 1, p. 198]{ZS}, the tensor product of two field extensions of $k$ is a domain if $k$ is algebraically closed. In this case, Theorem \ref{asso} can be reformulated as follows.

\begin{corollary}\label{closed}
Let  $k$ be an algebraically closed field. Then
\begin{enumerate}[\quad \rm(i)]
 \item $\Ass_R(M \otimes_k N) =\{\pp+\qq \mid \pp\in \Ass_A (M) \text{ and } \qq \in \Ass_B (N)\}$,
 \item $\Min_R(M \otimes_k N) =\{\pp+\qq \mid \pp\in \Min_A (M) \text{ and } \qq \in \Min_B (N)\}$.
\end{enumerate}
\end{corollary}


\section{Binomial expansion of symbolic powers} \label{sect_binomialexpansion}

Let $A$ and $B$ be two commutative Noetherian algebras over a field $k$ such that $R = A \otimes_k B$ is also a Noetherian ring.
Let $I$ and $J$ be nonzero proper ideals of $A$ and $B$, respectively.
We will use the same symbols $I,J$ for the extensions of $I,J$ in $R$.
The aim of this section is to prove that the symbolic power $(I+J)^{(n)}$ has a binomial expansion. \par

We shall need the following observations.

\begin{lemma} \label{HoaTam}
$I \cap J = IJ$.
\end{lemma}

\begin{proof}
Let $V$ and $W$ be two sets of elements of $A$ and $B$.
We denote by $V \otimes W$ the set of the elements $f \otimes g$, $f \in V$ and $g \in W$.
Choose bases $V$ and $W$ of the $k$-vector spaces $I$ and $J$ and extend them to bases $V^*$ and $W^*$ of $A$ and $B$, respectively.
Then $V \otimes W^*$ and $V^* \otimes W$ are bases of the $k$-vector spaces $I \otimes_k B$ and $A \otimes_k J$, respectively.
Since $V \otimes W^*$ and $V^* \otimes W$ are subsets of $V^* \otimes W^*$, which is a basis of $A \otimes_k B$, the vector space
$I \cap J = (I \otimes_k B) \cap (A \otimes_k J)$ is generated by the set
$(V \otimes W^*) \cap (V^* \otimes W) = V \otimes W.$
Since $IJ = I \otimes_kJ$ is also generated by $V \otimes W$, we conclude that $I \cap J = IJ$.
\end{proof}

Lemma \ref{HoaTam} was known before for polynomial ideals \cite[Lemma 1.1]{HT}.

\begin{lemma} \label{tensor}
Let $I'$ and $J'$ be subideals of $I$ and $J$, respectively. Then
$$(I/I') \otimes_k (J/J') \cong IJ /(IJ' + I'J).$$
\end{lemma}

\begin{proof} We have
\begin{align*}
(I/I') \otimes_k (J/J') & \cong \big((I \otimes_k J)/(I \otimes_kJ')\big)/\big((I' \otimes_k J)/(I'\otimes_kJ')\big).\\
& \cong (IJ/ IJ')/(I'J/I'J').
\end{align*}
By Lemma \ref{HoaTam},
$$I'J'=I'\cap J'=I'J\cap J' = I'J\cap IJ'.$$
From this it follows that
$$ (I/I') \otimes_k (J/J') \cong (IJ/ IJ')\big/\big((IJ' + I'J)/IJ'\big) \cong IJ \big/(IJ' + IJ').$$
\end{proof}

In the following, we will consider binomial sums of filtrations of ideals, which is defined as follows. \par

For simplicity, we call a sequence of ideals $\{Q_n\}_{n \ge 0}$ in $R$ a {\em filtration} if it satisfies the following conditions:
\begin{enumerate}
\item $Q_0 = R$,
\item $Q_1$ is a nonzero proper ideal,
\item $Q_n \supseteq Q_{n+1}$ for all $n \ge 0$.
\end{enumerate}

Examples of such filtrations are the ordinary powers $\{Q^n\}_{n \ge 0}$, the symbolic powers $\{Q^{(n)}\}_{n \ge 0}$,  and the integral closures of powers $\{\overline{Q^n}\}_{n \ge 0}$, where $Q$ is a nonzero proper ideal.

Let $\{I_i\}_{i \ge 0}$ and $\{J_j\}_{j \ge 0}$ be two filtrations of ideals in $A$ and $B$, respectively. For each $n\ge 0$, we define
$$Q_n := \sum_{i+j=n}I_iJ_j.$$
We call $Q_n$ the $n$-th \emph{binomial sum} of the filtrations $\{I_i\}_{i \ge 0}$ and $\{J_j\}_{j \ge 0}$.

Our next result shows that quotients of successive binomial sums have a nice decomposition.

\begin{proposition} \label{decomposition}
For any integer $n \ge 0$, there is an isomorphism
$$\displaystyle Q_n/Q_{n+1}
\cong \bigoplus_{i=0}^n (I_i/I_{i+1}) \otimes_k (J_{n-i}/J_{n-i+1}).$$
\end{proposition}

\begin{proof} First, we will show that
$$\displaystyle Q_n/Q_{n+1} \cong \bigoplus_{i=0}^n \big((I_iJ_{n-i} + Q_{n+1})/Q_{n+1}\big) .$$
For that it suffices to show that for $0\le i\le n$,
$$
(I_iJ_{n-i}+Q_{n+1})\cap \big(\sum\limits_{j\neq i}I_jJ_{n-j}+Q_{n+1}\big) \subseteq Q_{n+1}
$$
or, equivalently,
$$
I_iJ_{n-i} \cap \big(\sum\limits_{j\neq i}I_jJ_{n-j}+Q_{n+1}\big) \subseteq Q_{n+1}.
$$

We have
$$\sum\limits_{j\neq i}I_jJ_{n-j} + Q_{n+1} = \sum\limits_{j\neq i}I_jJ_{n-j} + \sum\limits_{0 \le j\le n+1}I_jJ_{n-j+1}
 \subseteq J_{n-i+1} + I_{i+1}$$
because $J_{n-j} \subseteq J_{n-i+1}$ for $j < i$ and $I_j\subseteq I_{i+1}$ for $j \ge i+1$.
On the other hand, every element of $J_{n-i+1} + I_{i+1}$ in $I_iJ_{n-i}$ is a sum of two elements,
one in $J_{n-i+1} \cap (I_{i+1}+I_iJ_{n-i})$ and the other in $I_{i+1} \cap (J_{n-i+1}+I_iJ_{n-i})$. Therefore,
\begin{align*}
&I_iJ_{n-i} \cap \big(\sum\limits_{0\le j\le n, j\neq i}I_jJ_{n-j}+Q_{n+1}\big)\\
&\subseteq J_{n-i+1} \cap (I_{i+1}+I_iJ_{n-i}) + I_{i+1} \cap (J_{n-i+1}+I_iJ_{n-i})\\
& \subseteq J_{n-i+1}\cap I_i + I_{i+1}\cap J_{n-i}=I_{i+1}J_{n-i}+I_iJ_{n-i+1} \subseteq Q_{n+1},
\end{align*}
where the equality holds thanks to Lemma \ref{HoaTam}.

The above inclusions also show that $I_iJ_{n-i} \cap Q_{n+1} = I_{i+1}J_{n-i}+I_iJ_{n-i+1}.$ Hence,
\begin{align*}
(I_iJ_{n-i} + Q_{n+1})/Q_{n+1} & \cong  I_iJ_{n-i} /\big(Q_{n+1}\cap I_iJ_{n-i}\big)\\
& \cong I_iJ_{n-i}/\big(I_{i+1}J_{n-i}+I_iJ_{n-i+1}\big)\\
& \cong (I_i/I_{i+1})\otimes_k (J_{n-i}/J_{n-i+1}),
\end{align*}
where the last isomorphism follows from Lemma \ref{tensor}. Therefore,
$$
Q_n/Q_{n+1} = \bigoplus_{i=0}^n \big((I_iJ_{n-i} + Q_{n+1})/Q_{n+1}\big) \cong \bigoplus_{i=0}^n (I_i/I_{i+1})\otimes_k (J_{n-i}/J_{n-i+1}).
$$
\end{proof}

We are now ready to prove the main result of this section.

\begin{theorem} \label{binomial}
$(I+J)^{(n)} = \displaystyle \sum_{i+j = n} I^{(i)} J^{(j)}.$
\end{theorem}

\begin{proof}
Consider the symbolic filtrations $\{I^{(i)}\}_{i \ge 0}$ and $\{J^{(j)}\}_{j \ge 0}$.
In this case, we have the $n$-th binomial sum
$$Q_n = \sum_{i+j=n}I^{(i)}J^{(j)}.$$

We shall first prove the inclusion $(I+J)^{(n)} \subseteq Q_n$.
For that we need to investigate the associated primes of $R/Q_n$.
It follows from the short exact sequence
$$0 \rightarrow  Q_{t-1}/Q_t \rightarrow R/Q_t \rightarrow R/Q_{t-1} \rightarrow 0,$$
that $\Ass_R(R/Q_t) \subseteq  \Ass_R(Q_{t-1}/Q_t) \cup \Ass_R(R/Q_{t-1}).$ Hence,
$$\Ass_R(R/Q_n) \subseteq \bigcup_{t=1}^n \Ass_R(Q_{t-1}/Q_t).$$
By Proposition \ref{decomposition}, we have
$$\Ass_R(Q_{t-1}/Q_t) = \bigcup_{i+j=t-1}\Ass_R\big(I^{(i)}/I^{(i+1)} \otimes_k J^{(j)}/J^{(j+1)}\big).$$
By Theorem \ref{asso}(ii), we have
$$\Ass_R\big(I^{(i)}/I^{(i+1)} \otimes_k J^{(j)}/J^{(j+1)}\big) = \bigcup_{\begin{subarray}{l} \pp\in \Ass_A(I^{(i)}/I^{(i+1)})\\
 \qq \in \Ass_B(J^{(j)}/J^{(j+1)})\end{subarray}}\Min_R(R/\pp+\qq).$$
Since $I^{(i)}/I^{(i+1)}$ and $J^{(j)}/J^{(j+1)}$ are ideals of $A/I^{(i+1)}$ and $B/J^{(j+1)}$, we have
\begin{align*}
\Ass_A(I^{(i)}/I^{(i+1)}) & \subseteq \Ass_A(A/I^{(i+1)}) = \Min_A(A/I^{(i+1)}) = \Min_A(A/I),\\
\Ass_B(J^{(j)}/J^{(j+1)} & \subseteq \Ass_B(B/J^{(j+1)}) = \Min_B(B/J^{(j+1)}) = \Min_B(B/J).
\end{align*}
Since $R/I+J = (A/I) \otimes_k (B/J)$, it follows from Theorem \ref{asso}(i) that
$$\Min_R(R/I+J) = \bigcup_{\begin{subarray}{l} \pp\in \Min_A(A/I)\\
 \qq \in \Min_B(B/J)\end{subarray}}\Min_R(R/\pp+\qq).$$
Therefore,
$$\Ass_R\big(I^{(i)}/I^{(i+1)} \otimes_k J^{(j)}/J^{(j+1)}\big)  \subseteq \Min_R(R/I+J).$$
So we get
\begin{align*}
\Ass_R(R/Q_n) & \subseteq \Min_R(R/I+J) =  \Min_R(R/(I+J)^n). \label{eq.Min}
\end{align*}
This shows that every associated prime of $Q_n$ is a minimal associated prime of $(I+J)^n$.
Since $Q_n \supseteq  \sum_{i+j = n} I^i J^j = (I+J)^n$,
it follows from the definition of symbolic powers that every primary component of $Q_n$ contains a primary component of $(I+J)^{(n)}$.
Therefore,  $Q_n \supseteq (I+J)^{(n)}$.

Now, we shall prove the converse inclusion $(I+J)^{(n)} \supseteq Q_n$.
Let $P$ be an arbitrary minimal associated prime of $(I+J)^n$.
Then $P$ is a minimal associated prime of $R/I+J = (A/I) \otimes_k(B/J)$.
Set $\pp = P \cap A$ and $\qq = P \cap B$. By Lemma \ref{contraction}(i),
$\pp$ and $\qq$ are minimal associated primes of $I$ and $J$.
Therefore, $(I^{(i)})_\pp = (I^i)_\pp$ and $(J^{(j)})_\qq = (J^j)_\qq$ for all $i, j \ge 0$.
This implies that
$$(I^{(i)})_\pp \otimes_k (J^{(j)})_\qq = (I^i)_\pp \otimes_k (J^j)_\qq.$$
Since $(I^{(i)} \otimes J^{(j)})_P$ and $(I^i \otimes J^j)_P$ are localizations of
$(I^{(i)})_\pp \otimes_k (J^{(j)})_\qq$ and $(I^i)_\pp \otimes_k (J^j)_\qq$ at a
prime ideal of $A_\pp \otimes B_\qq$, we get
$$(I^{(i)}J^{(j)})_P = (I^{(i)} \otimes_k J^{(j)})_P = (I^i \otimes J^j)_P = (I^iJ^j)_P$$
for all $i,j \ge 0$. Thus,
$$(Q_n)_P = \sum_{i+j = n} (I^{(i)}J^{(j)})_P = \sum_{i+j=n} (I^iJ^j)_P = (I+J)^n_P.$$
This shows that every primary component of $(I+J)^{(n)}$ contains a primary component of  $Q_n$.
Hence, $(I+J)^{(n)} \supseteq Q_n$.
\end{proof}

Theorem \ref{binomial} extends a result on squarefree monomial ideals of Bocci et al \cite[Theorem 7.8]{Many} to arbitrary ideals. It will play a key role in our investigation on invariants of $(I+J)^{(n)}$ in the next sections.

Moreover, Theorem \ref{binomial} yields the following criterion for the equality of symbolic and ordinary powers of $I+J$.

\begin{corollary} \label{equality}
Assume that  $I^t \neq I^{t+1}$ and $J^t \neq J^{t+1}$ for all $t \le n-1$.
Then
$(I+J)^{(n)} = (I+J)^n$ if and only if $I^{(t)} = I^t$ and $J^{(t)} = J^t$ for all $t \le n$.
\end{corollary}

\begin{proof}
Assume that $I^{(t)} = I^t$ and $J^{(t)} = J^t$ for all $t \le n$.
By Theorem \ref{binomial}, we have
$$(I+J)^{(n)} = \sum_{i+j=n} I^{(i)}J^{(j)} = \sum_{i+j=n} I^iJ^j = (I+J)^n.$$

Conversely, assume that $(I+J)^{(n)} = (I+J)^n$.
Since $(I+J)^{n-1}/(I+J)^n \subseteq R/(I+J)^n$, we have
\begin{align*}
\Ass_R((I+J)^{n-1}/(I+J)^n) & \subseteq \Ass_R(R/(I+J)^n)= \Min_R(R/(I+J)^n)\\
& = \Min_R(R/(I+J))= \Min_R\big((A/I)\otimes_k (B/J)\big).
\end{align*}
By Proposition \ref{decomposition}, we have
$$(I+J)^{n-1}/(I+J)^n = \bigoplus_{i+j=n-1} (I^i/I^{i+1}) \otimes_k (J^j/J^{j+1}).$$
Hence,
$$
\Ass_R\big((I+J)^{n-1}/(I+J)^n\big) = \mathop{\bigcup}_{i+j=n-1} \Ass_R\big((I^i/I^{i+1}) \otimes_k (J^j/J^{j+1})\big).$$ Therefore,
$$\Ass_R\big((I^i/I^{i+1}) \otimes_k (J^j/J^{j+1})\big) \subseteq \Min_R\big((A/I)\otimes_k (B/J)\big).$$
Since $I^i \neq I^{i+1}$ and $J^j \neq J^{j+1}$, we can apply Theorem \ref{asso} to get
$\Ass_A(I^i/I^{i+1}) \subseteq \Min_A(A/I)$ for $i\le n-1$. Similarly as in the proof of Theorem \ref{binomial}, we have
$$\Ass_A(A/I^t) \subseteq \bigcup_{i=0}^{t-1} \Ass_A(I^i/I^{i+1})\subseteq \Min_A(A/I).$$
This implies that $\Ass_A(A/I^t) = \Min_A(A/I)$.
Hence, $I^{(t)} = I^t$ for all $t \le n$.
By the same way, we can also show that $J^{(t)} = J^t$ for all $t \le n$.
\end{proof}

It is easy to see that the assumption of Corollary \ref{equality}  is satisfied if $A$ and $B$ are local rings or domains.
The following example shows that Corollary \ref{equality} does not hold if we remove the assumption that $I^t \neq I^{t+1}$ and $J^t \neq J^{t+1}$ for all $t \le n-1$.

\begin{example}
Let $A = k[x]/(x-x^2)$ and $I = xA$. Then $I^2 = I$.
Let $B = k[y,z,t]$ and $J = (y^4,y^3z,yz^3,z^4,y^2z^2t)$. Then $J^{(1)} = (y,z)^4 \neq J$ and $J^{(2)} = J^2 = (y,z)^8$.
However, $(I+J)^{(2)} = (I+J)^2 = (x,(y,z)^8)R$.
\end{example}

\begin{remark}
Corollary \ref{equality} shows that if $(I+J)^{(n)} = (I+J)^n$ then $(I+J)^{(t)} = (I+J)^t$ for all $t \le n-1$. However, for an arbitrary ideal $Q$ in a polynomial ring, $Q^{(n)} = Q^n$ does not imply $Q^{(t)} = Q^t$ for all $t \le n-1$. The ideal $J$ in the above example is such a case.
\end{remark}

We end this section by giving another interesting application of Theorem \ref{binomial}. Recall that for a nonzero proper homogeneous ideal $I$, $\alpha(I) := \min \{d ~|~ I_d \not= 0\}$ is the smallest degree of a nonzero element in $I$, and the \emph{Waldschmidt constant} of $I$ is defined by
$$\hat{\alpha}(I) := \displaystyle \lim_{m \rightarrow \infty} \frac{\alpha(I^{(m)})}{m}.$$
This limit exists and was first investigated by Waldschmidt  in complex analysis \cite{Wa}.
Since then, it has appeared in different areas of mathematics, e.g., in number theory,
algebraic geometry and commutative algebra \cite{Ch, EV, HaHu}.
The following consequence of Theorem \ref{binomial} extends a result on the Waldschmidt constant of squarefree monomial ideals \cite[Corollary 7.10]{Many} to arbitrary homogeneous ideals.

\begin{corollary} \label{cor.Waldschmidt}
Let $A$ and $B$ be standard graded polynomial rings over $k$. Let $I \subset A$ and $J \subset B$ be nonzero proper homogeneous ideals. Then
$$\hat{\alpha}(I+J) = \min \{\hat{\alpha}(I), \hat{\alpha}(J)\}.$$
\end{corollary}

\begin{proof} The proof goes in the same line of arguments as that of \cite[Corollary 7.10]{Many} replacing \cite[Theorem 7.8]{Many} by our more general statement in Theorem \ref{binomial}.
\end{proof}


\section{Depth and regularity of binomial sums}
\label{sect_depthreg}

Throughout this section, let $A = k[X]$ and $B = k[Y]$ be polynomial rings over an arbitrary field $k$ in different sets of variables.
Then $R := A \otimes_k B = k[X,Y]$. If $I \subset A$ and $J \subset B$ are homogeneous ideal, then their extensions in $R$ are also homogeneous. As before, we shall also denote these ideals by $I$ and $J$.

We shall need the following results of Hoa and Tam in \cite{HT}.

\begin{lemma} \label{HoaTam2} \cite[Lemmas 2.2 and 3.2]{HT}
Let $I \subseteq A$ and $J \subseteq B$ be nonzero proper homogeneous ideals. Then
\begin{enumerate}[\quad \rm(i)]
 \item $\depth R/IJ = \depth A/I + \depth B/J + 1$.
 \item $\reg R/IJ = \reg A/I + \reg B/J + 1$.
\end{enumerate}
\end{lemma}

Let $\{I_i\}_{i \ge 0}$ and $\{J_j\}_{j \ge 0}$ be filtrations of homogeneous ideals in $A$ and $B$, respectively.
Recall that the ideal
$$Q_n := \sum_{i+j = n} I_i J_j$$
is called the $n$-th binomial sum of these filtrations.
The aim of this section is to give bounds for the depth and the regularity of $R/Q_n$ in terms of  those of $I_i$ and $J_j$.

\begin{theorem}
\label{thm_depthreg_binomialproduct}
For all $n\ge 1$, we have
\begin{enumerate}[\quad \rm(i)]
\item $\depth R/Q_n\ge $
$$\min_{\begin{subarray}{l} i \in [1,n-1]\\ j \in [1,n] \end{subarray}} \left\{\depth A/I_{n-i}+\depth B/J_i+1, \depth A/I_{n-j+1}+\depth B/J_j\right\},$$
\item $\reg R/Q_n \le $
$$\max_{\begin{subarray}{l} i \in [1,n-1]\\ j \in [1,n] \end{subarray}} \left\{\reg A/I_{n-i}+\reg B/J_i +1, \reg A/I_{n-j+1}+\reg B/J_j \right\}.$$
\end{enumerate}
\end{theorem}

\begin{proof}
We shall only prove the bound for depth. The bound for regularity can be obtained in the same fashion.

For $t = 0, \dots, n$, set
$$P_{n,t}:= I_nJ_0 + I_{n-1}J_1 + \dots + I_{n-t}J_t.$$
Then $P_{n,t} = P_{n,t-1}+ I_{n-t}J_t$ for $1 \le t \le n$.
Since $P_{n,t-1} \subseteq I_{n-t+1}$, we have
$$P_{n,t-1}\cap I_{n-t}J_t \subseteq I_{n-t+1} \cap J_t = I_{n-t+1}J_t$$
by Lemma \ref{HoaTam}.
On the other hand, $I_{n-t+1}J_t \subseteq I_{n-t+1}J_{t-1} \subseteq P_{n,t-1}$ and $I_{n-t+1}J_t \subseteq I_{n-t}J_t$.
This implies that
$$P_{n,t-1} \cap I_{n-t}J_t = I_{n-t+1}J_t.$$
Hence, there is a short exact sequence
\[
0\longrightarrow R/I_{n-t+1}J_t \longrightarrow (R/P_{n,t-1})\oplus (R/I_{n-t}J_t) \longrightarrow R/P_{n,t} \longrightarrow 0.
\]
Therefore, we have
$$
\depth R/P_{n,t} \ge \min\left\{\depth R/P_{n,t-1}, \depth R/I_{n-t}J_t, \depth R/I_{n-t+1}J_t-1\right\}.
$$
We will use these bounds to successively estimate $\depth R/Q_n$ as follows.

For $t=n$, we have $P_{n,n} = Q_n$. Since $I_0J_n=J_n$,  we have
$\depth R/I_0J_n = \dim A+\depth B/J_n$.
Applying Lemma \ref{HoaTam2} to the product $I_1J_n$, we get
\begin{align*}
& \depth R/Q_n \ge\\
& \min\left\{\depth R/P_{n,n-1}, \dim A+\depth B/J_n, \depth A/I_1 + \depth B/J_n\right\}.
\end{align*}

For $t = n-1,\ldots,2$, by applying Lemma \ref{HoaTam2} to the products $I_{n-t}J_t$ and $I_{n-t+1}J_t$, we get
\begin{align*}
& \depth R/P_{n,t} \ge\\
& \min\left\{\depth R/P_{n,t-1}, \depth A/I_{n-t}\!  + \depth B/J_t\! +\! 1, \depth A/I_{n-t+1}\! +\! \depth B/J_t\right\}.
\end{align*}

For $t=1$, we have $\depth R/P_{n,0} =\depth A/I_n+\dim B$ because $P_{n,0}=I_nJ_0=I_n$.
Applying Lemma \ref{HoaTam2} to the product $I_nJ_1$ now yields
\begin{align*}
& \depth R/P_{n,1} \ge\\
& \min\left\{\depth A/I_n\! +\! \dim B, \depth A/I_{n-1}\! +\! \depth B/J_1\! +\! 1, \depth A/I_n\! +\! \depth B/J_1\right\}.
\end{align*}

Putting all these estimates for $\depth R/P_{n,t}$ together, we obtain
\begin{align*}
& \depth R/Q_n \ge \min \left\{\dim A+ \depth B/J_n, \depth A/I_n+\dim B, \right. \\
& \min_{\begin{subarray}{l} i \in [1,n-1]\\ j \in [1,n] \end{subarray}}   \{\depth A/I_{n-i}+\depth B/J_i +1, \depth A/I_{n-j+1}+\depth B/J_j\} \left. \right\} .
\end{align*}
Since $I_1 \neq (0)$, we have
$$\dim A +\depth B/J_n > \depth A/I_1+\depth B/J_n.$$
Since $J_1 \neq (0)$, we have
$$\depth A/I_n+\dim B > \depth A/I_n+\depth B/J_1.$$
The right-hand sides of the above inequalities are $\depth A/I_{n-j+1}+\depth B/J_j$ for $j = n,1$. Hence, we can remove the terms $\dim A +\depth B/J_n$ and $\depth A/I_n+\dim B$ from the estimate for $\depth R/Q_n$ to obtain that
\begin{align*}
& \depth R/Q_n\ge \\
& \min_{\begin{subarray}{l} i \in [1,n-1]\\ j \in [1,n] \end{subarray}} \left\{\depth A/I_{n-i}+\depth B/J_i +1, \depth A/I_{n-j+1}+\depth B/J_j \right\}.
\end{align*}
\end{proof}

\begin{remark}
Let $I \subset A$ and $J \subset B$ be nonzero homogeneous ideals.
If $I_i = I^i$ and $J_j = J^j$ for all $i, j \ge 0$, we have $Q_n = (I+J)^n$.
In this case, Theorem \ref{thm_depthreg_binomialproduct} recovers a main result of our previous work on depth and regularity of ordinary powers \cite[Theorem 2.4]{HTT}. As pointed out in \cite[Propositions 2.6 and 2.7]{HTT}, both terms on the right-hand side of these bounds are essential (i.e., are attainable). Hence, this is also the case for the two terms on the right-hand side of the bounds of Theorem \ref{thm_depthreg_binomialproduct}.
\end{remark}

If we consider the quotients $Q_n/Q_{n+1}$ instead of the quotient rings $R/Q_n$,
we can compute their depth and regularity explicitly in terms of those of quotients of successive $I_i$'s and $J_j$'s.

\begin{theorem}
\label{thm_depthreg_Qn/Qn+1}
For all $n\ge 1$, we have
\begin{enumerate}[\quad \rm(i)]
 \item $\displaystyle \depth Q_n/Q_{n+1}=\min_{i+j=n}\left\{\depth I_i/I_{i+1}+\depth J_j/J_{j+1} \right\}$,
 \item $\displaystyle \reg Q_n/Q_{n+1}=\max_{i+j=n}\left\{\reg I_i/I_{i+1}+\reg J_j/J_{j+1} \right\}$.
\end{enumerate}
\end{theorem}

\begin{proof} Proposition \ref{decomposition} gives
\[
Q_n/Q_{n+1}=\bigoplus_{i+j=n} (I_i/I_{i+1})\otimes_k (J_j/J_{j+1}).
\]
The desired conclusion now follows from \cite[Lemma 2.5]{HTT}, which expresses the depth and the regularity of a tensor product over $k$ in terms of those of the components.
\end{proof}

As a consequence of Theorem \ref{thm_depthreg_Qn/Qn+1}, we obtain bounds for the depth and the regularity of $R/Q_n$ in terms of those of quotients of successive $I_i$'s and $J_j$'s.

\begin{corollary}
For all  $n \ge 1$, we have
\begin{enumerate}[\quad \rm(i)]
\item $\displaystyle \depth R/Q_n \ge  \min_{i+j \le n-1}\{\depth I_i/I_{i+1} + \depth J_j/J_{j+1}\}$,
\item $\displaystyle \reg R/Q_n \le \max_{i+j\le n-1}\{\reg  I_i/I_{i+1} + \reg J_j/J_{j+1}\}.$
\end{enumerate}
\end{corollary}

\begin{proof}
Using the short exact sequences
$$0 \to Q_t/Q_{t+1} \to R/Q_{t+1} \to R/Q_t \to 0$$
for $t \le n-1$ we deduce that
\begin{align*}
\depth R/Q_n & \ge  \min_{t \le n-1} \depth Q_t/Q_{t+1},\\
\reg R/Q_n & \le \max_{t \le n-1} \reg Q_t/Q_{t+1}.
\end{align*}
Hence, the assertions follow from Theorem \ref{thm_depthreg_Qn/Qn+1}.
\end{proof}

If $\{I_i\}_{i \ge 0}$ and $\{J_j\}_{j \ge 0}$ are the filtrations of symbolic powers of two nonzero homogeneous ideals
$I \subset A$ and $J \subset B$, then $Q_n = (I+J)^{(n)}$ by Theorem \ref{binomial}.
For the sake of applications, we reformulate Theorem \ref{thm_depthreg_binomialproduct} and Theorem \ref{thm_depthreg_Qn/Qn+1} in this case.

\begin{theorem}\label{first bound}
For all $n \ge 1$, we have
\begin{enumerate}[\quad \rm(i)]
 \item $\depth R\big/(I+J)^{(n)} \ge$
$$\min_{\begin{subarray}{l} i \in [1,n-1]\\ j \in [1,n] \end{subarray}} \left\{\depth A/I^{(n-i)}  +  \depth B/J^{(i)} + 1, \depth A/I^{(n-j+1)}  +  \depth B/J^{(j)}\right\},$$
\item $\reg R\big/(I+J)^{(n)} \le$
$$\max_{\begin{subarray}{l} i \in [1,n-1]\\ j \in [1,n] \end{subarray}} \left\{\reg A/I^{(n-i)}  +  \reg B/J^{(i)} +  1, \reg A/I^{(n-j+1)}  +  \reg B/J^{(j)} \right\}.$$
\end{enumerate}
\end{theorem}

We shall see in the next section that the inequalities of Theorem \ref{first bound} are in fact equalities if $\chara(k) = 0$ or if $I$ and $J$ are monomial ideals.

\begin{theorem} \label{main equality}
For all  $n \ge 1$, we have
\begin{enumerate}[\quad \rm(i)]
\item $\displaystyle \depth (I+J)^{(n)}\big/(I+J)^{(n+1)} = \min_{i+j=n}\{\depth I^{(i)}/I^{(i+1)} + \depth J^{(j)}/J^{(j+1)}\}$,
\item $\displaystyle \reg (I+J)^{(n)}\big/(I+J)^{(n+1)} = \max_{i+j=n}\{\reg  I^{(i)}/I^{(i+1)} + \reg J^{(j)}/J^{(j+1)}\}.$
\end{enumerate}
\end{theorem}

Using Theorem \ref{main equality} we obtain the following criterion for the Cohen-Macaulayness of $R/(I+J)^{(n)}$.
Recall that a finite graded $R$-module $M$ is Cohen-Macaulay if $\depth M = \dim M$.

\begin{corollary} \label{cor.CM}
The following conditions are equivalent:
\begin{enumerate}[\quad \rm(i)]
\item $(I+J)^{(n-1)}\big/(I+J)^{(n)}$ is Cohen-Macaulay,
\item $R/(I+J)^{(i)}$ is Cohen-Macaulay for all $i \le n$,
\item $A/I^{(i)}$ and $B/J^{(i)}$ are Cohen-Macaulay for all $i \le n$,
\item $I^{(i)}/I^{(i+1)}$ and $J^{(i)}/J^{(i+1)}$ are Cohen-Macaulay for all $i \le n-1$.
\end{enumerate}
\end{corollary}

\begin{proof}
It is clear that
$$\dim (I+J)^{(n-1)}/(I+J)^{(n)} \le \dim R/(I+J)^{(n)} = \dim R/(I+J).$$
For any prime $P \in \Min_R(R/(I+J))$, we have
$$((I+J)^{(n-1)}/(I+J)^{(n)})_P = ((I+J)^{n-1}/(I+J)^{n})_P = 0$$
if and only if
$((I+J)^{n-1})_P = 0$ by Nakayama's lemma. But this could not happen because $R$ is a domain and $I+J \neq 0$.
Thus,
$$\dim R/(I+J) \le \dim (I+J)^{(n-1)}/(I+J)^{(n)}.$$
From this it follows that
$$\dim (I+J)^{(n-1)}/(I+J)^{(n)} = \dim R/(I+J) = \dim A/I + \dim B/J.$$
Similarly, we also have
$\dim I^{(i-1)}/I^{(i)} = \dim A/I$ and $\dim J^{(j-1)}/I^{(j)} = \dim B/J$ for all $i, j \ge 1$.\par

The above formulas imply
$$\dim (I+J)^{(n-1)}/(I+J)^{(n)} = \dim I^i/I^{(i+1)} + \dim J^{(i)}/J^{(i+1)}$$
for all $n,i,j \ge 1$. Using Theorem \ref{main equality}(i) we can show that
$$\depth (I+J)^{(n-1)}/(I+J)^{(n)} = \dim (I+J)^{(n-1)}/(I+J)^{(n)}$$
if and only if $\depth I^i/I^{(i+1)} = \dim I^i/I^{(i+1)}$ and $\depth J^{(i)}/J^{(i+1)} = \dim J^{(i)}/J^{(i+1)}$ for all $i \le n-1$. Thus, $(I+J)^{(n-1)}/(I+J)^{(n)}$ is Cohen-Macaulay if and only if $I^{(i)}/I^{(i+1)}$ and $J^{(i)}/J^{(i+1)}$ are Cohen-Macaulay for all $i \le n-1$. From this it follows that (i) $\Leftrightarrow$ (iv). In particular, (i) implies that $(I+J)^{(i)}/(I+J)^{(i+1)}$ is Cohen-Macaulay for all $i \le n-1$. \par

Note that a finite graded $R$-module $M$ is Cohen-Macaulay if and only if
$H_\mm^t(M) = 0$ for $t < \dim M$, where
$H_\mm^t(M)$ denotes the $t$-th local cohomology module of $M$ with respect to the maximal homogeneous ideal $\mm$ of $R$.
Using this characterization and the short exact sequence
$$0 \to (I+J)^{(i)}/(I+J)^{(i+1)} \to R/(I+J)^{(i+1)} \to R/(I+J)^{(i)} \to 0,$$
we deduce that  $R/(I+J)^{(i)}$ is Cohen-Macaulay for all $i \le n$ if and only if $(I+J)^{(i)}/(I+J)^{(i+1)}$ is Cohen-Macaulay for all $i \le n-1$. This proves the implication (i) $\Leftrightarrow$ (ii). \par

Similarly, $A/I^{(i)}$ and $B/J^{(i)}$ are Cohen-Macaulay for all $i \le n$ if and only if $I^{(i)}/I^{(i+1)}$ and $J^{(i)}/J^{(i+1)}$ are Cohen-Macaulay for all $i \le n-1$. This proves the equivalence (iii) $\Leftrightarrow$ (iv).
\end{proof}

The implication (i) $\Rightarrow$ (ii) is remarkable in the sense that the Cohen-Macaulay property of $(I+J)^{(n-1)}\big/(I+J)^{(n)}$ alone implies that of $R/(I+J)^{(t)}$ for all $t \le n$.
The following example shows  that this implication does not hold if we replace $I+J$ by an arbitrary homogeneous ideal.

\begin{example}
Take $A=k[x,y,z,t]$ and $I=(x^5,x^4y,xy^4,y^5,x^2y^2(xz+yt),x^3y^3)$.
Then $I$ is a $(x,y)$-primary ideal, $\dim A/I = 2$ and $I^{(1)}=I$.
It is clear that $z$ is a regular element of $A/I$.
Since the socle of $A/(I,z)$ contains the residue class of $x^2y^3$, we have $\depth A/I=1$.
Therefore, $A/I^{(1)} = A/I$ is not Cohen-Macaulay. On the other hand, we have
$I^2= (x,y)^{10} \cap \big(I^2+(z,t)\big).$
Hence, $I^{(2)}=(x,y)^{10}$.
Now, we can see that $z,t$ is a regular sequence of $I^{(1)}/I^{(2)}$.
Therefore, $I^{(1)}/I^{(2)}$ is Cohen-Macaulay.
\end{example}

We end this section with the following formulas for the case where one of the ideals $I$ and $J$ is generated by linear forms.
Though this case appears to be simple, these formulas had not been known before.
The proof is based on the binomial expansion of $(I+J)^{(n)}$.

\begin{proposition} \label{linear}
Assume that $J$ is generated by linear forms. Then
\begin{enumerate}[\quad \rm(i)]
\item $\displaystyle \depth R/(I+J)^{(n)} = \min_{i \le n}\{\depth A/I^{(i)} + \dim B/J\}$,
\item $\displaystyle \reg R/(I+J)^{(n)} = \max_{i \le n}\{\reg  A/I^{(i)} - i\} + n.$
\end{enumerate}
\end{proposition}

\begin{proof}
Assume that $B = k[y_1,\ldots,y_s]$.
Without restriction we may assume that $J = (y_1,\ldots,y_t)$, $t \le s$. Then $\dim B/J = s-t$.
Set $B' = k[y_1,\ldots,y_t]$, $J' = (y_1,\ldots,y_t)B'$, and $R' = A \otimes_k B'$.
It is clear that
\begin{align*}
\depth R/(I+J)^{(n)} & =  \depth R'/(I+J')^{(n)}+ s-t,\\
 \reg R/(I+J)^{(n)} & = \reg R'/(I+J')^{(n)}.
\end{align*}
Therefore, we only need to prove the case $t=s$.  \par
If $t=s=1$, we set $y=y_1$. Then $B = k[y]$ and $J = (y)$.
By Theorem \ref{binomial}, we have $(I,y)^{(n)} = \sum_{i=0}^n I^{(i)}y^{n-i}$.
If we write $R = \oplus_{i \ge 0} Ay^i$, then $R/(I,y)^{(n)} = \oplus_{i \le n} (A/I^{(i)})y^{n-i}$.
From this it follows that
\begin{align*}
\depth R/(I,y)^{(n)} & =  \min_{i\le n}\depth A/I^{(i)} ,\\
 \reg R/(I,y)^{(n)} & = \max_{i \le n}\{\reg A/I^{(i)}-i\} +n.
\end{align*}\par
If $t = s > 1$, we set $A' = A[y_1,\ldots,y_{s-1}]$ and $I' = (I,y_1,\ldots,y_{s-1})A'$.
Using induction we may assume that
\begin{align*}
\depth A'/(I')^{(n)} & =  \min_{i\le n}\depth A/I^{(i)},\\
 \reg A'/(I')^{(n)} & = \max_{i \leq n}\{\reg A/I^{(i)}-i\}+n.
\end{align*}
Since $I+J = (I',y_s)$, we have
\begin{align*}
\depth R/(I+J)^{(n)} & =  \min_{i\le n}\depth A'/(I')^{(i)} = \min_{i\le n} \depth A/I^{(i)},\\
 \reg R/(I+J)^{(n)} & = \max_{i \leq n}\{\reg A'/(I')^{(i)} -i\}+n = \max_{i \le n}\{ \reg A/I^{(i)}-i\}+n.
\end{align*}
\end{proof}


\section{Splitting conditions}
\label{sect_splittings}

The aim of this section is to show that the inequalities of Theorem \ref{first bound} are equalities
if $\chara(k)=0$ or if $I$ and $J$ are monomial ideals.
Our main tool is the following notion which allows us to compute the depth and the regularity of a sum of ideals in terms of those of the summands and their intersection. \par

Let $R$ be a polynomial ring over a field $k$.
Let $P, I, J$ be nonzero homogeneous ideals of $R$ such that $P=I+J$.
The sum $P = I+J$ is called a {\it Betti splitting} if the Betti numbers of $P,I,J, I \cap J$ satisfy the relation
\[
\beta_{i,j}(P)=\beta_{i,j}(I)+\beta_{i,j}(J)+\beta_{i-1,j}(I\cap J)
\]
for all $i\ge 0$ and $j \in \ZZ$.

This notion was introduced by Francisco, H\`a and Van Tuyl \cite{FHV}.
It generalizes the notion of splittable monomial ideals of Eliahou and Kervaire \cite{EK}.
The following result explains why it is useful to have a Betti splitting.

\begin{lemma}
\label{lem_depthreg_Bettisplit} \cite[Corollary 2.2]{FHV}
Let $P = I+J$ be a Betti splitting. Then
\begin{enumerate}[\quad \rm(i)]
\item $\depth R/P =\min\{\depth R/I, \depth R/J, \depth R/I\cap J -1\},$
\item $\reg R/P = \max\{\reg R/I, \reg R/J, \reg R/I\cap J -1\}.$
\end{enumerate}
\end{lemma}

A Betti splitting can be characterized by the following property.
We say that a homomorphism $\phi: M \to N$ of graded $R$-modules is {\it Tor-vanishing} if  $\Tor^R_i(k,\phi)=0$ for all $i\ge 0$.
This notion was due to Nguyen and Vu \cite{NgV}.

\begin{lemma}  \protect{\cite[Proposition 2.1]{FHV}}
\label{lem_Bettisplit}
The following conditions are equivalent:
\begin{enumerate}[\quad \rm(i)]
\item The decomposition $P=I+J$ is a Betti splitting.
\item The inclusion maps $I\cap J\to I$ and $I\cap J \to J$ are Tor-vanishing.
\end{enumerate}
\end{lemma}

We say that a filtration of ideals $\{P_n\}_{n \ge 0}$ in $R$
is a {\it Tor-vanishing filtration} if for all $n\ge 1$, the inclusion map $P_n\to P_{n-1}$ is Tor-vanishing, i.e. $\Tor^R_i(k,P_n) \to \Tor^R_i(k,P_{n-1})$ is the zero map for all $i \ge 0$.

The following result shows that the inequalities of Theorem \ref{thm_depthreg_binomialproduct} are equalities if the given filtrations are Tor-vanishing.

\begin{theorem}
\label{thm_Bettisplitting_binomialproduct}
Let $\{I_i\}_{i \ge 0}$ and $\{J_j\}_{j \ge 0}$ be Tor-vanishing filtrations in $A$ and $B$.
Let $Q_n = \sum_{i+j=n}I_iJ_j$. Then
\begin{enumerate}[\quad \rm(i)]
\item $\depth R/Q_n =$
$$\min_{\begin{subarray}{l} i \in [1,n-1]\\ j \in [1,n] \end{subarray}} \left\{\depth A/I_{n-i} +\depth B/J_i+1, \depth A/I_{n-j+1} + \depth B/J_j\right\},$$
\item $\reg R/Q_n =$
$$\max_{\begin{subarray}{l} i \in [1,n-1]\\ j \in [1,n] \end{subarray}} \left\{\reg A/I_{n-i}+\reg B/J_i+1, \reg A/I_{n-j+1}+\reg B/J_j\right\}.$$
\end{enumerate}
\end{theorem}

\begin{proof}
We shall only prove the  formula for depth. The formula for regularity can be proved similarly.

For $t = 0, \dots, n$, set
$$P_{n,t} :=  I_nJ_0 + I_{n-1}J_1 + \cdots + I_{n-t}J_t.$$
Then $P_{n,0} = I_n$,  $P_{n,n} = Q_n$, and $P_{n,t} = P_{n,t-1}+I_{n-t}J_t$ for $t \ge 1$.
We will compute $\depth R/Q_n$ by successively computing $\depth R/P_{n,t}$.
To do that we will first show that $P_{n,t} = P_{n,t-1}+I_{n-t}J_t$ is a Betti splitting for $1\le t\le n$.

We have seen in the proof of Theorem \ref{thm_depthreg_binomialproduct} that
$$P_{n,t-1} \cap I_{n-t}J_t = I_{n-t+1}J_t.$$
Note that $I_{n-t+1}J_t = I_{n-t+1} \otimes_kJ_t$ and $I_{n-t}J_t = I_{n-t}\otimes_kJ_t$.
By the hypothesis, the inclusion map $I_{n-t+1} \to I_{n-t}$ is Tor-vanishing.
Since the tensor product with $J_t$ is an exact functor, the inclusion map $I_{n-t+1}J_t \to I_{n-t}J_t$ is also Tor-vanishing.
Similarly, the inclusion map $I_{n-t+1}J_t\to I_{n-t+1}J_{t-1}$ is Tor-vanishing.
Since $I_{n-t+1}J_t \subseteq I_{n-t+1}J_{t-1} \subseteq P_{n,t-1}$, the inclusion map $I_{n-t+1}J_t \to P_{n,t-1}$ is also Tor-vanishing.
By Lemma \ref{lem_Bettisplit}, it follows that $P_{n,t} = P_{n,t-1}+I_{n-t}J_t$ is a Betti splitting.

Now we can apply Lemma \ref{lem_depthreg_Bettisplit} to obtain
\[
\depth R/P_{n,t} =\min\left\{\depth R/P_{n,t-1}, \depth R/I_{n-t}J_t, \depth R/I_{n-t+1}J_t-1\right\}.
\]
Proceeding as in the proof of Theorem \ref{thm_depthreg_binomialproduct} we will get the desired conclusion.
\end{proof}

Tor-vanishing filtrations can be found by the following result of Ahangari Maleki \cite{AM} in the case $\chara(k)=0$.
Let $\partial(Q)$ denote the the ideal generated by the partial derivatives of the generators of an ideal $Q$. Using the chain rule for taking partial derivatives, it is not hard to see that $\partial(Q)$ does not depend on the chosen generators of $Q$.

\begin{lemma}
\label{lem_maleki} \cite[Propostion 3.5]{AM}
Assume that $\chara(k)=0$. Let $Q \subseteq Q'$ be homogeneous ideals such that $\partial(Q) \subseteq Q'$.
Then the inclusion map $Q \to Q'$ is Tor-vanishing.
\end{lemma}

\begin{proposition}
\label{mapoftors_symbolicpower}
Let $I$ be a nonzero proper homogeneous ideal in $A$.
If $\chara(k)=0$ then $\partial(I^{(n)}) \subseteq I^{(n-1)}$ for all $n \ge 1$.
Therefore, the filtration $\{I^{(n)}\}_{n \ge 0}$ is Tor-vanishing.
\end{proposition}

\begin{proof}
Let $\Ass^*(I) := \bigcup_{n=1}^{\infty}\Ass_A(I^n)$.
This is a finite set due to a classical result of Brodmann \cite{Br1}.
Set
$$L=\bigcap_{P\in \Ass^*(I) \setminus \Min(I)}P,$$
where $L =A$ if $\Ass^*(I) = \Min(I)$.
It follows from the definition of symbolic powers that $I^{(m)} = \cup_{t \ge 0}I^m:L^t$ for all $m \ge 0$.

Let $t \ge 1$ be an integer such that $I^{(n)}L^t \subseteq I^n$.
For any $f \in I^{(n)}$ and $g \in L^{t+1}$, we have $fg \in I^n$.
Let $\partial_x$ denote the partial derivative with respect to an arbitrary variable $x$ of $A$.
We have
$$\partial_x(f)g+f\partial_x(g) = \partial_x(fg) \in \partial(I^n) \subseteq I^{n-1}.$$
Since $\partial_x(g)  \in \partial(L^{t+1}) \subseteq L^t$, we have
$$f\partial_x(g) \in fL^t \subseteq I^n.$$
Therefore, $\partial_x(f)g = \partial_x(fg) - f\partial_x(g) \in I^{n-1}$. Hence, $\partial_x(f) \in I^{n-1}:L^{t+1} \subseteq I^{(n-1)}$.
So we can conclude that $\partial(I^{(n)}) \subseteq I^{(n-1)}$.
\end{proof}

Now, we can show that the bounds in Theorem \ref{first bound} are actually the exact values of the depth and the regularity of $R/(I+J)^{(n)}$ if $\chara(k)=0$.

\begin{theorem}
\label{thm_depthreg_char0monomial}
Let $I$ and $J$ be nonzero proper homogeneous ideals in $A$ and $B$.
Assume that $\chara(k)=0$. Then for all $n\ge 1$, we have
\begin{enumerate}[\quad \rm(i)]
\item $\depth R/(I+J)^{(n)}=$
$$\min_{\begin{subarray}{l} i \in [1,n-1]\\ j \in [1,n] \end{subarray}} \left\{\depth  A/I^{(n-i)}  +  \depth B/J^{(i)} + 1, \depth A/I^{(n-j+1)} + \depth B/J^{(j)}\right\},$$
\item $\reg R/(I+J)^{(n)} =$
$$\max_{\begin{subarray}{l} i \in [1,n-1]\\ j \in [1,n] \end{subarray}} \left\{\reg A/I^{(n-i)} +  \reg B/J^{(i)} + 1, \reg A/I^{(n-j+1)} + \reg B/J^{(j)}\right\}.$$
\end{enumerate}
\end{theorem}

\begin{proof}
By Proposition \ref{mapoftors_symbolicpower}, the filtrations $\{I^{(i)}\}_{i \ge 0}$ and $\{J^{(j)}\}_{j \ge 0}$ are Tor-vanishing. Therefore, the conclusion follows from Theorem \ref{binomial} and Theorem \ref{thm_Bettisplitting_binomialproduct}.
\end{proof}

The following example shows that both equalities of Theorem \ref{thm_depthreg_char0monomial} may fail if one of the base rings is not regular, even if $\chara(k)=0$.

\begin{example}
\label{ex_nonpolynomial}
By abuse of notations, we will denote the residue class of an element $f$ by $f$ itself.
Let $S=\QQ[a,b,c,d]$. Put $A = S/(a^4,a^3b,ab^3,b^4,a^3c-b^3d)$ and $I =(a^2-b^2)A$.
Computations with Macaulay2 \cite{GS} show that $I^{(1)}= I, I^{(2)}= I^2= (a^2b^2)A$.
Hence, there are isomorphisms of $S$-modules
\begin{align*}
A/I^{(1)} &\cong S/(a^2-b^2,a^4,a^3b,a^3c-b^3d),\\
A/I^{(2)} &\cong S/(a^4,a^3b,a^2b^2,ab^3,b^4,a^3c-b^3d).
\end{align*}
This implies that
\begin{align*}
\depth(A/I^{(1)}) &=\depth(A/I^{(2)})=1,\\
\reg(A/I^{(1)})   &=\reg(A/I^{(2)})=3.
\end{align*}
Let $B =\QQ[x,y,z,t]$ and $J=(x^5,x^4y,xy^4,y^5,x^2y^2(xz^6-yt^6),x^3y^3) \subseteq B$.
Computations with Macaulay2 give
$J^{(1)}=J=(x^5,x^4y,xy^4,y^5,x^2y^2(xz^6-yt^6),x^3y^3)$, $J^{(2)}=(x,y)^{10}$, and
\begin{align*}
\depth(B/J^{(1)})&=1,\ \depth(B/J^{(2)})=2,\\
\reg(B/J^{(1)})  &=10,\ \reg(B/J^{(2)})=9.
\end{align*}
Let $R = A \otimes_kB$ and $Q=I+J \subseteq R$. Then
$$R=\QQ[a,b,c,d,x,y,z,t]/(a^4,a^3b,ab^3,b^4,a^3c-b^3d). $$
$$Q= (a^2-b^2,x^5,x^4y,xy^4,y^5,x^2y^2(xz^6-yt^6),x^3y^3)R.$$
By Theorem \ref{binomial}, we have
\begin{align*}
Q^{(2)} &=I^{(2)}+I^{(1)}J^{(1)}+J^{(2)}\\
        &=(a^2b^2,(a^2-b^2)(x^5,x^4y,xy^4,y^5,x^2y^2(xz^6-yt^6),x^3y^3,(x,y)^{10})R.
\end{align*}
It can be checked that
\begin{align*}
\depth(R/Q^{(2)}) &=3 > 2=\depth(A/I^{(2)})+\depth(B/J^{(1)}),\\
\reg(R/Q^{(2)})   &=13 < 14=\reg(A/I^{(1)})+\reg(B/J^{(1)})+1.
\end{align*}
Hence, both equalities of Theorem \ref{thm_depthreg_char0monomial} fail in this example.
\end{example}

The Tor-vanishing of the symbolic powers in Proposition \ref{mapoftors_symbolicpower} can be considered as a higher order generalization of the inclusion $I^{(n)} \subseteq \mm I^{(n-1)}$, which was proved by Eisenbud and Mazur for $\chara(k) = 0$ \cite[Proposition 13]{EM} and for $I$ being a monomial ideal \cite[Proposition 9]{EM}.

Eisenbud and Mazur \cite{EM} conjectured that $I^{(2)} \subseteq \mm I^{(1)}$ if $I$ is a prime ideal in a power series ring over a field of characteristic zero (see \cite{HaHu} for similar conjectures on higher symbolic powers).
They also showed that this conjecture has a negative answer in positive characteristic. Using their result we can construct the following example, which shows that Proposition \ref{mapoftors_symbolicpower} does not hold in positive characteristic.

\begin{example}
\label{ex_Eisenbud-Mazur}
Let $R=(\mathbb Z/2)[x,y,z,t]$ and $S=(\mathbb Z/2)[u]$. Consider the kernel $L$ of the map $R \to S$ given by
$$x \mapsto u^4, y\mapsto u^6, z \mapsto u^7, t \mapsto u^9.$$
Clearly, $L$ is a prime ideal. Hence, $L^{(1)}=L$.
Computations with Macaulay2 \cite{GS} show that
$L$ is minimally generated by the following polynomials:
$$g_1=x^3+y^2, g_2=yz+xt, g_3=x^2y+z^2, xz^2+t^2, x^2z+yt,xy^2+zt.$$
The map $\Tor^R_0(k,L^{(2)})\to \Tor^R_0(k,L)$ is not zero. In fact, for $f=x^3y-y^3-xz^2+t^2$, we have $x^2f=g_1g_3+g_2^2 \in L^2$, which shows that $f\in L^{(2)}$. On the other hand, we have $f \notin \mm_R L$, where $\mm_R =(x,y,z,t)$. \par
While the ideal $L$ is not homogeneous in the standard grading, it is weighted homogeneous by setting
$\deg x=4, \deg y=6, \deg z=7, \deg t=9.$
Now we use the polarization trick of J. McCullough and I. Peeva \cite{MP} to transform $L$ into a homogeneous ideal in the standard grading. Let $A=(\mathbb Z/2)[x_1,x_2,x_3,x_4,y_1,y_2,y_3,y_4]$.
Consider the homogeneous map $\phi: R\to A$ given by
$$x\mapsto x_1^3y_1, y\mapsto x_2^5y_2, z\mapsto x_3^6y_3, t \mapsto x_4^8y_4.$$
Let $I$ be the extension of $L$ to $A$. Then one can check with Macaulay2  that $I$ is a prime ideal of $A$, and more importantly, $I$ is homogeneous in the standard grading of $A$. Let $\mm_A =(x_1,\ldots,x_4,y_1,\ldots,y_4)$. Then the relation $x^2f=g_1g_3+g_2^2$ shows that $\phi(f) \in I^{(2)} \setminus \mm_A I$. Hence, the map $I^{(2)}\to I^{(1)}$ is not Tor-vanishing. Using similar arguments and \cite[Example, p. 200]{EM}, we even have similar examples in any positive characteristic.
\end{example}

We could not use  Example \ref{ex_Eisenbud-Mazur} to construct any counterexample to the conclusion of Theorem \ref{thm_depthreg_char0monomial} in positive characteristics.  We expect that Theorem \ref{thm_depthreg_char0monomial} holds regardless of the characteristic of $k$. \par

Our next main result shows that this is the case  for monomial ideals.
We shall first collect alternatives for Lemma \ref{lem_maleki} and Proposition \ref{mapoftors_symbolicpower} in this case, and then present the result in Theorem \ref{monomial}.

For a monomial ideal $Q$ we denote by $\partial^*(Q)$ the ideal generated by elements of the form $f/x$, where $f$ is a minimal monomial generator of $Q$ and $x$ is a variable dividing $f$.

\begin{lemma}  \label{NgV} \cite[Proposition 4.4 and Lemma 4.2]{NgV}
Let $Q \subseteq Q'$ be monomial ideals such that $\partial^*(Q) \subseteq Q'$.
Then the inclusion map $Q \to Q'$ is Tor-vanishing.
\end{lemma}

\begin{proposition}
\label{function_monomial}
Let $I$ be a nonzero proper monomial ideal in $A$.
Then $\partial^*(I^{(n)}) \subseteq I^{(n-1)}$ for all $n \ge 1$.
Therefore, the filtration $\{I^{(n)}\}_{n \ge 0}$ is Tor-vanishing.
\end{proposition}

\begin{proof}
Let $Q_1,\ldots,Q_s$ be the primary components of $I$ associated to the minimal primes of $I$.
Note that $Q_1,\ldots,Q_s$ are monomial ideals.
By \cite[Lemma 3.1]{HHT}, we have
\[
I^{(n)}= Q_1^n \cap \cdots \cap Q_s^n.
\]
It is clear that $\partial^*(Q^n)=\partial^*(Q)Q^{n-1} \subseteq Q^{n-1}$ for any monomial ideal $Q$.
Therefore,
$$\partial^*(I^{(n)}) \subseteq \partial^*(Q_1^n) \cap \cdots \cap \partial^*(Q_s^n) \subseteq Q_1^{n-1} \cap \cdots \cap Q_s^{n-1}= I^{(n-1)}.$$
By Lemma \ref{NgV}, this implies that $\{I^{(n)}\}_{n \ge 0}$ is Tor-vanishing.
\end{proof}

\begin{theorem} \label{monomial}
Let $I$ and $J$ be nonzero proper monomial ideals in $A$ and $B$.
Then the equalities of Theorem \ref{thm_depthreg_char0monomial} hold regardless of the characteristic of $k$.
\end{theorem}

\begin{proof}
By Proposition \ref{function_monomial}, both the filtrations $\{I^{(i)}\}_{i \ge 0}$ and $\{J^{(j)}\}_{j \ge 0}$
are Tor-vanishing. Therefore, the conclusion follows from Theorem \ref{binomial} and Theorem \ref{thm_Bettisplitting_binomialproduct}.\end{proof}

Our results on Tor-vanishing have some interesting consequences on the relationship between
the depth and the regularity of $A/I^{(n-1)}, A/I^{(n)}$, and $I^{(n-1)}/I^{(n)}$.

\begin{proposition}
Let $I$ be a nonzero proper homogeneous ideal in $A$.	
Assume that $\chara(k)=0$ or $I$ is a monomial ideal. Then for any $n\ge 1$,
\begin{enumerate}[\quad \rm (i)]
\item $\depth  I^{(n-1)}/I^{(n)} = \min\{\depth A/I^{(n-1)}+1, \depth A/I^{(n)}\},$
\item $\reg  I^{(n-1)}/I^{(n)} = \max\{\reg I^{(n-1)}, \reg I^{(n)}-1\}.$
\end{enumerate}
\end{proposition}

\begin{proof}
By Proposition \ref{mapoftors_symbolicpower} and Proposition \ref{function_monomial}, the inclusion map $I^{(n)} \to I^{(n-1)}$ is Tor-vanishing, i.e.
$\Tor^A_i(k, I^{(n)}) \to \Tor^A_i(k, I^{(n-1)})$ is the zero map for all $i$. From the exact sequence
$$0\to I^{(n)} \to I^{(n-1)} \to I^{(n-1)}/I^{(n)} \to 0,$$
it follows that the long exact sequence of Tor splits into short exact sequences
$$0\to \Tor^A_{i+1}(k, I^{(n-1)})  \to \Tor^A_{i+1}(k, I^{(n-1)}/I^{(n)}) \to \Tor^A_i(k, I^{(n)}) \to 0.$$
Using the characterization of the depth and the regularity by Tor we can easily deduce the conclusion.
\end{proof}



\end{document}